\documentclass[12pt]{amsart}
\usepackage[utf8]{inputenc}
\usepackage{amsmath}
\usepackage{amssymb}
\usepackage{amsthm}
\usepackage[margin=1.1in]{geometry}
\usepackage{csquotes}
\usepackage{xcolor}
\theoremstyle{remark}
\newtheorem{theorem}{Theorem}
\newtheorem{lemma}{Lemma}
\newtheorem{corollary}{Corollary}
\newtheorem{prop}{Proposition}

\numberwithin{equation}{section}

\usepackage{hyperref}

\begin{document}

\title[Intersective Polynomials]{Polynomials with factors of the form $ (x^{q} - a)$ with roots modulo every integer}


\author[B. Mishra]{Bhawesh Mishra}
\address{Department of Mathematical Sciences, The University of Memphis, Memphis, TN 38152}
\email{\url{bmishra1@memphis.edu}}
\thanks{}


\date{\today}

\dedicatory{}

\commby{}
\maketitle
\begin{abstract}
Given an odd prime $q$, a natural number $l$ and non-zero $q$-free integers $a_{1}, a_{2}, \ldots, a_{l}$, none of which are equal to $1$ or $-1$, we give necessary and sufficient conditions for the polynomial $\prod_{j=1}^{l} (x^{q} - a_{j})$ to have roots modulo every positive integer. Consequently: (i) if $l \leq q$ and none of $a_{1}, a_{2}, \ldots, a_{l}$ is a perfect $q^{th}$ power, then the polynomial $\prod_{j=1}^{l} (x^{q} - a_{j})$ fails to have roots modulo some positive integer; $(ii)$ For every $l\in\mathbb{N}$, and every $(c_{j})_{j=1}^{l}\in\big(\mathbb{F}_{q}\setminus\{0\}\big)^{l}$, the polynomial $\prod_{j=1}^{l} (x^{q} - a_{j})$ has roots modulo every positive integer if and only if $\prod_{j=1}^{l} (x^{q} - \text{rad}_{q}\big(a_{j}^{c_{j}}\big)))$ has roots modulo every positive integer. Here $\text{rad}_{q}(a_{j})$ denotes the $q$-free part of the integer $a_{j}$.
\end{abstract}

\section{Introduction}
\subsection{Motivation.} We are interested in polynomials $f(x) \in\mathbb{Z}[x]$ such that \begin{equation}
    f(x) \equiv 0 \hspace{1mm} (\text{mod } m) \text{ is solvable for every positive integer } m. \label{defn1}
\end{equation} 
A polynomial with integer root clearly has root modulo every positive integer. The interesting case is when a polynomial $f(x)\in\mathbb{Z}[x]$ satisfies \eqref{defn1} but does not have a linear factor. Polynomials that have roots modulo every positive integer are of interest in additive combinatorics.

Recall that the \textit{upper density} of a subset $T$ of integers is defined as \[\Bar{d}(T) = \limsup_{N\rightarrow\infty} \frac{|T \cap \{-N, \ldots, -1, 0, 1, \ldots, N\}|}{2N+1}.\]
Then, a set $S \subseteq \mathbb{Z}$ is said to be an \textit{intersective set} if given any $T \subseteq \mathbb{Z}$ with $\Bar{d}(T) > 0$, one has $ S \cap (T - T) \not\subseteq \{0\}$. In other words, an intersective set is called so because it intersects non-trivially with the set of differences \[T - T := \{t_{1} - t_{2} : t_{1}, t_{2} \in T\},\] of every positive upper density subset $T$ of integers. A polynomial $f(x) \in\mathbb{Z}[x]$ is said to be an \textit{intersective polynomial} if the set $\{f(x) : x\in\mathbb{Z}\}$ of its values is an intersective set.

S\'ark\"ozy and Furstenberg, independently proved in \cite{Sa1} and \cite{Fu} respectively, that the polynomial $f(x) = x^{2}$ is intersective. In other words, they showed that for every $S \subset\mathbb{Z}$ with $\Bar{d}(S) > 0$, $(S-S)$ contains a perfect square. The methods in \cite{Sa1} can be modified to show that for any $i \in\mathbb{N}$, the polynomial $f(x) = x^{i}$ is intersective \cite{Sa2}. Kamae and Mend\'es-France, in \cite{KaMF}, first proved that a polynomial $f(x) \in\mathbb{Z}[x]$ is intersective if and only if $f$ satisfies \eqref{defn1}. In other words, if a polynomial satisfies \eqref{defn1}, then for every $S \subseteq \mathbb{Z}$ with $\Bar{d}(S) > 0$, we have $\{a, a+f(n) \} \subseteq S$ for some $a \in S$ and $n \in\mathbb{Z}$. Bergelson, Leibman and Lesigne obtained a far reaching extension of this result in \cite{BLL}, a special instance of which implies that for a polynomial $f(x) \in\mathbb{Z}[x]$, every subset of positive upper density in $\mathbb{Z}$ contains arbitrarily long polynomial progressions of the form $$a, a+f(n), a+2f(n), \ldots, a+kf(n)$$ if and only if $f(x)$ satisfies \ref{defn1}. 

It is well known that if $p$ and $q$ are distinct odd primes then the polynomial \[p(x) = (x^{2} - p) (x^{2} - q) (x^{2} - pq)\] satisfies \ref{defn1} if and only if $p$ is a square modulo $q$ and $q$ is a square modulo $p$ (see \cite[pp 139--140]{Jones} or \cite[Corollary 1]{MishraAMM}). This example was generalized Hyde and Spearman in \cite{HS} in the following manner. Let $c$ and $d$ be square-free integers not equal to $1$, let $c_{1} = \frac{c}{\text{gcd}(c,d)}$ and $d_{1} = \frac{d}{\text{gcd}(c,d)}$. Hyde and Spearman obtained necessary and sufficient conditions for the polynomial 
\begin{equation}
    q(x) = (x^{2} - c) (x^{2} - d) (x^{2} - c_{1}d_{1}) \label{intro1}
\end{equation}
to have a root modulo every integer but have no rational root \cite{HS}. In \cite{MishraAMM}, the author generalized Hyde and Spearman's result by obtaining necessary and sufficient conditions for the polynomial \[g(x) = (x^{2} - a_{1}) (x^{2} - a_{2}) \cdots (x^{2} - a_{l})\] to have roots modulo every positive integer. Here, $a_{1}, a_{2}, \ldots, a_{l}$ are distinct, non-zero square-free integers such that none of them are equal to $1$. Clearly, $g(x)$ cannot have a rational root since each $a_{i}$ is square-free and not equal to $1$. This article is about intersective polynomials that are of the form \[(x^{q}-a_{1})(x^{q}-a_{2})\ldots (x^{q}-a_{l}),\] for an odd prime $q$ and $q$-free integers $a_{1}, a_{2}, \ldots, a_{l}$, none of which are $1$ or $-1$ (see Theorem \ref{mainresult}). We will obtain necessary and sufficient conditions for such polynomials to be intersective.  

\subsection{Notation}
Throughout this article, $q$ will be used to denote an odd prime. An integer $a$ is called $q$-free if for every prime $p$, $p^{q} \nmid a$. We will denote the finite field with $q$ elements by $\mathbb{F}_{q}$. Given a prime $p$, an integer $\nu\geq 1$ and $a \in\mathbb{Z}$, we denote by $p^{\nu} \mid\mid a$ the fact that $p^{\nu}$ is the highest power of $p$ dividing $a$. The set of primes is denoted by $\mathbb{P}$. Given $S \subset\mathbb{P}$, we define its relative asymptotic density to be \[d_{\mathbb{P}}(S) := \lim_{n\rightarrow\infty} \frac{|S \cap \{1, 2, \ldots, n\}|}{|\mathbb{P} \cap \{1, 2, \ldots, n\}|},\] if the limit exists. We will say that $A\subseteq\mathbb{Z}$ contains a $q^{th}$ power residue modulo almost every prime if the set 
\[ \big\{p\in\mathbb{P} : A \text{ contains a } q^{th} \text{ power residue modulo } p \big\} \] has relative asymptotic density $1$. Recall that if a number field $K$ contains the complex $q^{th}$ root of unity $\zeta_{q}$, then for every prime ideal $\mathfrak{p}$ of $K$ coprime to $q$ and every $\mathfrak{p}$-adic unit $\alpha \in K$, we define the $q^{th}$ power residue symbol $\big(\alpha|\mathfrak{p}\big)_{q}$ to be the unique $q^{th}$ root of unity $\zeta_{q}^{j}$ such that 
\[\alpha^{\frac{\text{Norm}(\mathfrak{p})-1}{q}} \equiv \zeta_{q}^{j} \hspace{1mm} (\text{mod } \mathfrak{p}).\] Whenever $\big(\alpha|\mathfrak{p}\big)_{q}$ is defined and $\big(\alpha|\mathfrak{p}\big)_{q} \neq + 1$, $\alpha$ is not a $q^{th}$ power modulo $\mathfrak{p}$. 

\subsection{Adjustment within Positive $q$-free Class.}
Given an integer $b$ that is not a perfect $q^{th}$ power and has the unique prime factorization (with sign) $\pm \prod_{i=1}^{m} p_{i}^{a_{i}}$, the $q$-free part of $b$ can be defined as \[ \text{rad}_{q}(b) := \pm \prod_{i=1}^{m} p_{i}^{a_{i} (\text{mod }q)}.\] By convention, $-1$ and $1$ are $q$-free under this convention. An integer $b$ is a $q^{th}$ power modulo a prime $p$ not dividing $b$ if and only if $|b|$ is a $q^{th}$ power modulo $p$. This is because $-1$ is always a $q^{th}$ power. Therefore, given a set $A = \{a_{j}\}_{j=1}^{l} \subset\mathbb{Z}\setminus\{0\}$, as long as we are concerned with $q^{th}$ power residues modulo primes $p \nmid \prod_{j=1}^{l} a_{j}$ in $A$, one could replace the set $A$ by \[\big\{ \text{rad}_{q}(|a_{j}|) \big\}_{j=1}^{l}.\] For instance, when $q = 3$ one can replace the set $\{24, -104, 54\}$ by $\{3, 13, 2\}$ because $24 = 2^{3} \cdot 3$, $104 = 2^{3} \cdot 13$, and $54 = 3^{3} \cdot 2$. We will use this fact in the formulation of the Proposition \ref{Mainresult2} of the article. 
    
\subsection{Linear Coverings of a Vector Space.} \label{covering}
Let $V$ be a vector space of dimension $k \geq 2$ over a field $K$. A linear covering of $V$ is a collection of proper subspaces $\{W_{i}\}_{i\in I}$ such that $V = \bigcup_{i \in I} W_{i}$. The linear covering number of a vector space $V$, denoted by $\#$ LC$(V)$, is the minimum cardinality of a linear covering of $V$. We will use the following fact about $\#$ LC$(V)$, which is the part of the main result proved in \cite{Clark}.
\begin{prop}
For every vector space $V$ over $\mathbb{F}_{q}$, of dimension $\geq 2$, \#LC(V) = q + 1. \label{ILC}
\end{prop}

\subsection{Finite subsets of $q$-free integers and hyperplanes in a vector space.}\label{hyperplanes}
Let $A = \{a_{j}\}_{j=1}^{l} \subset\mathbb{Z}\setminus\{0,1\}$ be a finite set of $q$-free integers. Let $p_{1}, p_{2}, \ldots, p_{k}$ be all the primes that divide $\prod_{j=1}^{l} a_{j}$ and let $\nu_{ij} \geq 0$ such that $p_{i}^{\nu_{ij}} \mid\mid \text{rad}_{q}(|a_{j}|)$ for every $1 \leq i \leq k$ and every $1 \leq j \leq l$. Consider the equations of hyperplanes in $\mathbb{F}_{q}^{k}$ defined as \[\sum_{i=1}^{k} \nu_{ij} x_{i} = 0 \text{ for } 1 \leq j \leq l,\] in the variables $x_{1}, x_{2}, \ldots, x_{k}$. We call the set \[\Bigg\{ \Big\{ \big(x_{i}\big)_{i=1}^{k} \in\mathbb{F}_{q}^{k} : \sum_{i=1}^{k} \nu_{ij} x_{i} = 0 \Big\} : 1 \leq j \leq l \Bigg\},\] the set of hyperplanes in $\mathbb{F}_{q}^{k}$ associated with $A = \{a_{1}, a_{2}, \ldots, a_{l}\}$. Our main result is the following theorem. 
\begin{theorem}
Let $q$ be an odd prime and let $\{a_{j}\}_{j=1}^{l} \subseteq \mathbb{Z}\setminus\{-1,1\}$ be a finite set of non-zero $q$-free integers. Let $p_{1}, p_{2}, \ldots, p_{k} (k \geq 2)$ be all the primes dividing $\prod_{j=1}^{l} a_{j}$ and let $\nu_{ij} \geq 0$ such that $p_{i}^{\nu_{ij}} \mid\mid a_{j}$ for every $1 \leq i \leq k$ and $1 \leq j \leq l$. Then, the polynomial \[f(x) = (x^{q}-a_{1}) (x^{q} - a_{2}) \cdots (x^{q} - a_{l})\] is intersective if and only if following conditions are satisfied:
\begin{enumerate}
    \item The set of hyperplanes \[\Bigg\{ \Big\{ \big(x_{i}\big)_{i=1}^{k} \in\mathbb{F}_{q}^{k} : \sum_{i=1}^{k} \nu_{ij} x_{i} = 0 \Big\} : 1 \leq j \leq l \Bigg\}\] associated with the set $\{a_{1}, a_{2}, \ldots, a_{l}\}$ forms a linear covering of $\mathbb{F}_{q}^{k}$. \vspace{1mm}
    
    \item One of the $a_{i}$ is a $q^{th}$ power modulo $q^{q}$.\vspace{1mm}
    
    \item For every $ 1 \leq i \leq k$, there exists $j \in\{1, 2, \ldots, l\}$ such that $p_{i} \nmid a_{j}$ and $a_{j}$ is a $q^{th}$ power modulo $p_{i}$. 
\end{enumerate}
\label{mainresult}
\end{theorem}

Note that we stipulate $k \neq 1$ in Theorem \ref{mainresult} because one can easily verify that the polynomial $(x^{q}-p_{1})(x^{q} - p_{1}^{2}) \cdots (x^{q} - p_{1}^{q-1})$ is not intersective. Since $p_{1}$ is not a perfect $q^{th}$ power, there exist infinitely many primes $p$ not dividing $p_{1}$ such that $p_{1}$ is not a $q^{th}$ power modulo $p$. This follows from an application of the Grunwald-Wang theorem (see \cite[Theorem 1,pp. 75]{ArtinTate}). Then, it follows that none of the elements in the set $\{p_{1}, p_{1}^{2}, \ldots, p_{1}^{q-1}\}$ is a $q^{th}$ power $p$.  In other words, the polynomial $(x^{q}-p_{1})(x^{q} - p_{1}^{2}) \cdots (x^{q} - p_{1}^{q-1})$ does not have a root modulo such $p$.

\section{Some Elementary Results}
To prove our main result, we will use the following characterization of finite subsets of integers that contain a $q^{th}$ power modulo almost every prime. (see \cite[Theorem 1]{MishraFFA}). 

\begin{prop}
Let $q$ be an odd prime and let $A = \{a_{1}, a_{2}, \ldots, a_{l}\}\subset \mathbb{Z}\setminus\{0\}$ not containing a perfect $q^{th}$ power. Let $p_{1}, p_{2}, \ldots, p_{k}$ $(k \geq 2)$ be all the primes dividing elements of $\{\text{rad}_{q}(|a_{j}|) : 1 \leq j \leq l\}$ such that $p_{i}^{\nu_{ij}} \mid\mid \text{rad}_{q}(|a_{j}|)$ for every $1 \leq i \leq k$ and every $1 \leq j \leq l$. Then, the following three conditions are equivalent:
\begin{enumerate}
    \item The set $A$ contains a $q^{th}$ power modulo almost every prime. \vspace{1mm}
    
    \item For every prime $p\not\in\{q, p_{1}, p_{2}, \ldots, p_{k}\}$, the set $A$ contains a $q^{th}$ power modulo $p$. 
    
    \item The set of hyperplanes \[\Big\{ \sum_{i=1}^{k} \nu_{ij} x_{i} = 0 : 1 \leq j \leq l \Big\}\] associated with $A$, contains a linear covering of $\mathbb{F}_{q}^{k}$.
\end{enumerate}\label{Mainresult2}
\end{prop}
We will also state and prove two elementary lemmas that we shall repeatedly use in the proof of Theorem \ref{mainresult}. 
\begin{lemma}
Let $a \in\mathbb{Z}$ be a $q$-free integer and $p$ be an odd prime such that $p^{\nu} \mid\mid a$ for some $1 \leq \nu \leq q-1$. Then, $a$ cannot be a $q^{th}$ power modulo $p^{\nu+1}$. \label{Elementary1}
\end{lemma}

\begin{proof}
Assume to the contrary that $x_{0}^{q} \equiv a \hspace{1mm} (\text{mod } p^{\nu+1})$ for some integer $x_{0}$, i.e., $p^{\nu+1} \mid (x_{0}^{q} - a)$. Since $p^{\nu} \mid a$ and $p^{\nu} \mid p^{\nu+1} \mid (x_{0}^{q} - a)$, we have that $p^{\nu} \mid x_{0}^{q}$. and hence $p \mid x_{0}^{q}$. However, $p \mid x_{0}^{q}$ implies $p^{q} \mid x_{0}^{q}$, which together with $p^{\nu+1} \mid (x_{0}^{q} - a)$ gives that $p^{\nu+1} \mid a$ because $\nu+1 \leq q$. However, $p^{\nu+1} \mid a$ is a contradiction to the assusmption $p^{\nu} \mid\mid a$. Therefore, $a$ cannot be a $q^{th}$ power modulo $p^{\nu+1}$.
\end{proof}

\begin{lemma}
Let $q$ and $p$ be prime numbers, $k,l \in\mathbb{N}$ and $f(x) =  \prod_{j=1}^{l} (x^{q} - a_{i})$ for integers $a_{1}, a_{2}, \ldots, a_{l}$. If $x^{q} - a_{i} \equiv 0 \hspace{1mm} (\text{mod}\hspace{1mm} p^{k})$ is not solvable for any $1 \leq i \leq l$ then, $f(x) \equiv 0 \hspace{1mm} (\text{mod}\hspace{1mm} p^{kl})$ is not solvable.  \label{Elementary2}
\end{lemma}

\begin{proof}
For the sake of contradiction, assume that $f(x) \equiv 0 \hspace{1mm} (\text{mod}\hspace{1mm} p^{kl})$ is solvable for some $x_{0} \in\mathbb{Z}$, i.e., $p^{kl} \mid (x_{0}^{q} - a_{1}) \cdots (x_{0}^{q} - a_{l})$. Then we must have $p^{k} \mid (x_{0}^{q} - a_{j})$, for some $j \in \{1, 2, \ldots , l\}$. 

However, $p^{k} \mid x_{0}^{q} - a_{j}$ implies that $x_{0}^{q} - a_{j} \equiv 0 \hspace{1mm} (\text{mod}\hspace{1mm} p^{k})$ is solvable, a contradiction to the fact that $x^{q} - a_{i} \equiv 0 \hspace{1mm} (\text{mod}\hspace{1mm} p^{k})$ is not solvable for any $i$. Therefore, we have the result. 
\end{proof}
Now, we are ready to prove Theorem \ref{mainresult}. 
\section{Proof of Theorem \ref{mainresult}}
\subsection{Proof of necessity.} The necessity of the condition $(1)$ immediately follows from Proposition \ref{Mainresult2}. Now assume that condition $(2)$ of Theorem \ref{mainresult} is not satisfied, i.e., none of the $a_{i}$ is a $q^{th}$ power modulo $q^{q}$. Then, Lemma \ref{Elementary2} implies that $f(x) \equiv 0 \hspace{1mm} (\text{mod } q^{al})$ is not solvable. 

Now assume that the condition $(3)$ is not satisfied by $a_{1}, a_{2}, \ldots, a_{l}$. Then, there exists $1 \leq j_{0} \leq l$ and a prime $p \mid a_{j_{0}}$ such that for every $i \in\{1, 2, \ldots, l\}\setminus\{j_{0}\}$ we have following two cases:
\begin{enumerate}
    \item $p$ divides $a_{i}$. In this case, if $p^{\nu} \mid\mid a_{i}$, then $a_{i}$ is not a $q^{th}$ power modulo $p^{\nu+1}$. Since $\nu \leq q-1$, Lemma \ref{Elementary1} implies that  $a_{i}$ is not a $q^{th}$ power modulo $p^{q}$.
    
    \item $a_{i}$ is not a $q^{th}$ power modulo $p$, and consequently $a_{i}$ is not a $q^{th}$ power modulo $p^{q}$.
\end{enumerate}
Similarly, since $p \mid a_{j_{0}}$ we also have that $a_{j_{0}}$ is not a $q^{th}$ power modulo $p^{q}$, which follows from Lemma \ref{Elementary1}.

So far, we have that for every $i \in\{1, 2, \ldots, l\}$ the congruence $x^{q} - a_{i} \equiv 0 \hspace{1mm} (\text{mod } p^{q})$ is not solvable. Therefore, using Lemma \ref{Elementary2} we have that $f(x) \equiv 0 \hspace{1mm} (\text{mod } p^{ql})$ is not solvable. We have shown that if any of the conditions listed in Theorem \ref{mainresult} is not satisfied then, the polynomial $f(x)$ fails to have root modulo some integer. In other words, $f(x)$ does not satisfy \ref{defn1}.  

\subsection{Proof of sufficiency.} For the proof of this part, we will use the Hensel's lemma which implies that for $g(x) \in\mathbb{Z}[x]$ if there exists $a \in\mathbb{Z}$ such that $|g(a)|_{p} < |g^{\prime} (a)|^{2}_{p}$, then for every $i \in\mathbb{N}$ there exists $\alpha\in\mathbb{Z}$ such that $g(\alpha) \equiv 0 \hspace{1mm} (\text{mod } p^{i})$. Here, $|a|_{p}$ denotes the $p$-adic absolute value of $a$. For instance, if $p \mid g(a)$ and $p \nmid g^{\prime}(a)$, then the conclusion of the Hensel's lemma holds. Readers can find more on the Hensel's lemma in \cite[Proposition 3.5.2]{FrJar}. 

Assume that each condition in Theorem \ref{mainresult} is satisfied. We will show that $f(x) \equiv 0 \hspace{1mm}(\text{mod}\hspace{1mm} p^{b})$ is solvable for every prime $p$ and every integer $b \geq 1$. Three cases arise as follows:

\begin{enumerate}

\item If $p\not\in\{p_{1}, p_{2}, \ldots, p_{k}, q\}$, then Proposition \ref{Mainresult2} implies that one of the $a_{i}$, say $a_{i_{0}}$, is a $q^{th}$ power modulo $p$. In other words, for $g(x) = x^{q} - a_{i_{0}}$ there exists $b_{i_{0}}$ such that $g(b_{i_{0}}) = b_{i_{0}}^{q} - a_{i_{0}} \equiv 0 \hspace{1mm}(\text{mod } p)$ has a solution. 

If $p \mid b_{i_{0}}$, then using the fact that $p \mid (b_{i_{0}}^{q} - a_{i_{0}})$ we have that $p \mid a_{i_{0}}$. However, $p \mid a_{i_{0}}$ implies that $p \in\{p_{1}, p_{2}, \ldots, p_{k}\}$, which is a contradiction. Therefore, $p \nmid b_{i_{0}}$. 

The fact that $p\nmid b_{i_{0}}$, along with $p \neq q$, implies that $p \nmid g^{\prime}(b_{i_{0}}) = qb_{i_{0}}^{q-1}$. Since $p \mid g(b_{i_{0}})$ but $p \nmid g^{\prime} (b_{i_{0}})$, it follows from the Hensel's lemma that for every $b \geq 1$, the congruence $g(x) \equiv 0 \hspace{1mm}(\text{mod } p^{b})$.\vspace{1mm}

\item On the other hand, if $p$ is a prime different from $q$ such that $p \mid \prod_{j=1}^{l} a_{j}$, then $p \mid a_{j}$ for some $j \in\{1, 2, \ldots, l\}$. Then, condition $(3)$ in Theorem \ref{mainresult} implies that there exists $a_{i}$ with $p \nmid a_{i}$ such that $a_{i}$ is a $q^{th}$ power modulo $p$. In this case, $g(b_{i}) = (b_{i}^{q} - a_{i}) \equiv 0 \hspace{1mm} (\text{mod } p)$ for some integer $b_{i}$. 

Similar to that in the previous case, an application of the Hensel's lemma for the polynomial $(x^{q} - a_{i})$ again implies that $a_{i}$ is a $q^{th}$ power modulo $p^{b}$ for every $b \geq 1$. 

\item If $p = q$, then the condition $(2)$ of Theorem \ref{mainresult} implies that one of the $a_{i}$ is a $q^{th}$ power modulo $q^{q}$, i.e., $x_{0}^{q} - a_{r} \equiv 0 \hspace{1mm} (\text{mod } q^{q})$ for some $r \in\{1, 2, \ldots, l\}$ and some integer $x_{0}$. If we denote $g(x) := x^{q} - a_{r}$, then we see that $q^{q} \mid g(x_{0})$, i.e.,  $|g(x_{0})|_{q} \leq q^{-q}$. Here $|y|_{q}$ denotes the $q$-adic valuation of $y$. 

On the other hand, we claim that $|g^{\prime}(x_{0})|_{q} = q^{-1}$. Note that $q \nmid x_{0}$. Otherwise, $q \mid x_{0}$ along with $q^{q} \mid (x_{0}^{q} - a_{r})$ gives that $q^{q} \mid a_{r}$ contradicting the fact that $a_{r}$ is $q$-free. Since, $q \nmid x_{0}$ we have that $q \mid\mid qx_{0}^{q-1}$, i.e., $|g^{\prime}(x_{0})|_{q} = q^{-1}$. Therefore, \[|g^{\prime}(x_{0})|^{2}_{q} = q^{-2} > q^{-q} \geq |g(x_{0})|_{q} \] because $q$ is an odd prime. Hence, the Hensel's lemma implies that $g(x) \equiv 0 \hspace{1mm} (\text{mod } q^{b})$, and hence $f(x) \equiv 0 \hspace{1mm} (\text{mod } q^{b})$, is solvable for every $b \geq 1$. 
\end{enumerate}
So far, we have proved that for any prime $p$ and any integer $b \geq 1$ the congruence $f(x) \equiv 0 \hspace{1mm} (\text{mod } p^{b})$ is solvable. Now, it follows from the Chinese Remainder Theorem that the polynomial $f$ satisfies equation \ref{defn1}. 

\section{Corollaries and Examples}
Our first corollary shows that we need at least $q+1$ many distinct factors of the form $(x^{q} - a_{i})$ for the polynomial $f(x) = \prod (x^{q} - a_{i})$ to be intersective without having integer roots. 
\begin{corollary}
Let $q$ be a prime and let $l\in\mathbb{N}$ such that $l \leq q$. If $\prod_{j=1}^{l} (x^{q} - a_{j}) \in\mathbb{Z}[x]$ is intersective, then there exists $1 \leq j_{0} \leq l$ such that $a_{j_{0}}$ is a perfect $q^{th}$ power. \label{corlowerbound}
\end{corollary}

\begin{proof}
For the sake of contradiction, assume that for some $l \leq q$, $\prod_{j=1}^{l} (x^{q} - a_{j})$ is intersective and the set $A = \{a_{j}\}_{j=1}^{l}$ does not contain a perfect $q^{th}$ power. In particular, the set $A$ contains a $q^{th}$ power modulo almost every prime. Then, by Proposition \ref{Mainresult2}, the set of $l$ hyperplanes associated with $\{a_{j}\}_{j=1}^{l}$ must form a linear covering of $\mathbb{F}_{q}^{k}$. Since by Proposition \ref{ILC} a linear covering of $\mathbb{F}_{q}^{k}$ must have cardinality at least $q+1$, we must have $l \geq q+1$, which is a contradiction to $l \leq q$. Therefore, there exists $1 \leq j_{0} \leq l$ such that $a_{j_{0}}$ is a perfect $q^{th}$ power. 
\end{proof}

Our next corollary shows that the intersectivity of a polynomial $\prod_{j=1}^{l} (x^{q} - a_{j})$ is invariant under exponentiation of integers $a_{j}$ by non-zero elements of $\mathbb{F}_{q}$. 

\begin{corollary}
Let $q$ be an odd prime, $\{a_{1}, a_{2}, \ldots, a_{l}\}\in\mathbb{Z}\setminus\{-1,1\}$ be a finite subset of $q$-free integers and $(c_{j})_{j=1}^{l}\in\big(\mathbb{F}_{q}\setminus\{0\}\big)^{l}$. Then the polynomial \[(x^{q} - a_{1}) (x^{q} - a_{2}) \cdots (x^{q} - a_{l}) \] is intersective if and only if the polynomial \[\big(x^{q} - \text{rad}_{q}(a_{1}^{c_{1}})\big) \big(x^{q} - \text{rad}_{q}(a_{2}^{c_{2}})\big) \cdots \big(x^{q} - \text{rad}_{q}(a_{l}^{c_{l}})\big) \] is intersective. \label{corexponentiation}
\end{corollary}

\begin{proof}
Let $p_{1}, p_{2}, \ldots, p_{k}$ be all the primes dividing $\prod_{j=1}^{l} a_{j}$ and let $\nu_{ij} \geq 0$ be such that $p_{i}^{\nu_{ij}} \mid\mid a_{j}$ for every $1 \leq i \leq k$ and $1 \leq j \leq l$. For every $1 \leq j \leq l$, since $a_{j} = \prod_{i=1}^{k} p_{i}^{\nu_{ij}}$, we have that $a_{j}^{c_{j}} = \prod_{i=1}^{k} p_{i}^{c_{j}\nu_{ij}}$. Now, the corollary follows from Theorem \ref{mainresult} and the following three observations:
\begin{enumerate}
    \item Then, the set of hyperplanes associated with the set $\{a_{j}^{c_{j}}\}_{j=1}^{l}$ is  \[H_{1}:= \Bigg\{ \Big\{ \big(x_{i}\big)_{i=1}^{k} \in\mathbb{F}_{q}^{k} : \sum_{i=1}^{k} c_{j}\nu_{ij} x_{i} = 0 \Big\} : 1 \leq j \leq l \Bigg\},\] which is the same as the set of hyperplanes  \[H_{2} := \Bigg\{ \Big\{ \big(x_{i}\big)_{i=1}^{k} \in\mathbb{F}_{q}^{k} : \sum_{i=1}^{k} \nu_{ij} x_{i} = 0 \Big\} : 1 \leq j \leq l \Bigg\}\] associated with the set $\{a_{j}\}_{j=1}^{l}$.\vspace{1mm}
    
    \item For any prime $p \in\{p_{1}, p_{2}, \ldots, p_{k}\}$, a non-zero $q$-free integer $a$, with $p\nmid a$, is a $q^{th}$ power modulo $p$ if and only if $\text{rad}_{q}(a^{c})$ is a $q^{th}$ power modulo $p$ for every $c \in\mathbb{F}_{q}\setminus\{0\}$. Clearly, if $a$ is a $q^{th}$ power modulo $p$, then so is $\text{rad}_{q}(a^{c})$. On the other hand, assume that $a$ is not a $q^{th}$ power modulo $p$. Then, $\big(a|p\big)_{q} = \zeta_{q}^{j}$ for some $j \nmid q$. Therefore, we have \[\big(\text{rad}_{q}(a^{c})|p\big)_{q} = (\zeta_{q}^{j})^{c} \neq 1\] because $c\in\mathbb{F}_{q}\setminus\{0\}$ and $j \nmid q$. Therefore, $\text{rad}_{q}(a^{c})$ is not a $q^{th}$ power modulo $p$ either. \vspace{1mm}
    
    \item A non-zero $q$-free integer $a$ is a $q^{th}$ power modulo $q^{q}$ if and only if $\text{rad}_{q}(a^{c})$ is a $q^{th}$ power modulo $q^{q}$ for every $c \in\mathbb{F}_{q}\setminus\{0\}$. First, we note that if $a$ is a non-zero $q$-free integer and $q \mid a$, then both $a$ and $\text{rad}_{q}(a^{c})$ cannot be a $q^{th}$ power modulo $q^{q}$ as a consequence of Lemma \ref{Elementary1}. 
    
    Now assume that $q \nmid a$ and that $a^{c} = \text{rad}_{q}(a^{c}) \cdot r^{q}$ for some integer $r$ with $q \nmid r$. Assume that $a$ is a $q^{th}$ power modulo $q^{q}$ and $c \in\mathbb{F}_{q}\setminus\{0\}$. Then, for some integer $y$, we have \[q^{q} \mid (y^{q} - a) \mid \big( (y^{c})^{q} - a^{c}) = \big( (y^{c})^{q} - \text{rad}_{q}(a^{c}) \cdot r^{q} \big) \text{ and hence, }\] \[q^{q} \mid \Big( \big(\frac{y^{c}}{r}\big)^{q} - \text{rad}_{q}(a^{c}) \Big) \text{ since } q \nmid r.\] Here, $\frac{y^{c}}{r} = y^{c}r^{-1}$, where $r^{-1}$ denotes the inverse of $r \hspace{1mm} (\text{mod } q^{2})$. Therefore, $\text{rad}_{q}(a^{c})$ is also a $q^{th}$ power modulo $q^{q}$. 
    
    On the other hand, assume that $\text{rad}_{q}(a^{c})$ is a $q^{th}$ power modulo $q^{q}$, i.e., for some integer $z$, $q^{q} \mid \big( z^{q} - \text{rad}_{q}(a^{c})\big)$. Therefore, we have \[q^{q} \mid r^{q}\big(z^{q} - \text{rad}_{q}(a^{c})\big) = \big((zr)^{q} - a^{c}\big).\] In other words, we have $(zr)^{q} \equiv a^{c} \hspace{1mm} (\text{mod } q^{q})$, and hence $\big((zr)^{1/c}\big)^{q} \equiv a \hspace{1mm} (\text{mod } q^{q})$, which shows that $a$ is a $q^{th}$ power modulo $q^{q}$. Note that since $q \nmid a$, $q \nmid z$, and $q \nmid r$, therefore $zr$ and $(zr)^{1/c}$ is defined modulo $q^{q}$.
\end{enumerate}
\end{proof}

\subsection{Examples}
Now, we will use Theorem \ref{mainresult}, to produce illustrative examples of families of intersective polynomials for $q = 3$ and $q = 5$. The analogous examples for $q = 2$ are given in \cite{MishraAMM} and \cite{MishraIntegers}.
\subsection{q = 3.}
\begin{corollary}
Let $p_{1}, p_{2}$ be two distinct primes different from $3$. Then, the polynomial \[f_{1}(x) = (x^{3} - p_{1}) (x^{3} - p_{2}) (x^{3} - p_{1}p_{2}) (x^{3} - p_{1}^{2}p_{2})\] is intersective if and only if the following conditions are satisfied:
\begin{itemize}
    \item Some element of the set $\{p_{1}, p_{2}, p_{1}p_{2}, p_{1}^{2}p_{2}\}$ is a cube modulo $27$. \vspace{1mm}
    
    \item $p_{1}$ is a cube modulo $p_{2}$ and $p_{2}$ is a cube modulo $p_{1}$. 
\end{itemize} \label{cor1}
\end{corollary}

\begin{proof}
We note that the two conditions listed in the statement of the corollary are exactly the conditions $(2)$ and $(3)$ in the statement of Theorem \ref{mainresult}, for the prime $q = 3$ and the cube-free integers $a_{1} = p_{1}$, $a_{2} = p_{2}$, $a_{3} = p_{1}p_{2}$ and $a_{4} = p_{1}^{2}p_{2}$. Therefore, it suffices to show that $a_{1}, a_{2}, a_{3}, a_{4}$ satisfy the condition $(1)$ in Theorem \ref{mainresult}. Note that $p_{1}, p_{2}$ are the primes dividing $\prod_{j=1}^{4} a_{j}$ and 
\begin{multline*}
    \\ \nu_{11} = 1, \nu_{12} = 0, \nu_{13} = 1, \nu_{14} = 2 \\ \nu_{21} = 0, \nu_{22} = 1, \nu_{23} = 1, \nu_{24} = 1 \\
\end{multline*}
Therefore, the set of hyperplanes in $\mathbb{F}_{3}^{2}$ associated with $\{a_{1}, a_{2}, a_{3}, a_{4}\}$ is \[H_{1}:= \Big\{ (x_{1}, x_{2})\in\mathbb{F}_{3}^{2}: x_{1} = 0, x_{2} = 0, x_{1} + x_{2} = 0  \text{ and } 2x_{1} + x_{2} = 0.\Big\}.\] One can easily see that every $(x_{1}, x_{2}) \in\mathbb{F}_{3}^{2}$ satisfies one of the equations of hyperplanes in $H_{1}$. In other words, the set of hyperplanes $H_{1}$ form a linear covering of $\mathbb{F}_{3}^{2}$. The corollary is now established. 
\end{proof}
Following are some concrete examples of intersective polynomials obtained from Corollary \ref{cor1}. 
\begin{enumerate}
    \item The polynomial \[f(x) =  (x^{3} - 7) (x^{3} - 251) (x^{3} - 7 \times 251) (x^{3} - 7^{2} \times 251)\] is intersective as a result of applying Corollary \ref{cor1} for $p_{1} = 7$ and $p_{2} = 251$. Note that, \[251 \equiv 6 \equiv 3^{3} \hspace{1mm} (\text{mod } 7) \text{ and } 251 \equiv 8 \equiv 2^{3} \hspace{1mm} (\text{mod } 27). \] Furthermore, $7$ is a cube modulo $251$ because $251 \equiv 2 \hspace{1mm} (\text{mod } 3)$ and every non-zero residue class is a cube modulo any prime equivalent to $2$ modulo $3$.  
    
    \item The polynomial \[f(x) =  (x^{3} - 7) (x^{3} - 2141)  (x^{3} - 7 \times 2141 ) (x^{3} - 7^{2} \times 2141)\] is intersective as a result of applying Corollary \ref{cor1} for $p_{1} = 7$ and $p_{2} = 2141$. Note that, \[2141 \equiv 6 \equiv 3^{3} \hspace{1mm} (\text{mod } 7) \text{ and } 2141 \equiv 8 \equiv 2^{3} \hspace{1mm} (\text{mod } 27). \] Furthermore, $7$ is a cube modulo $2141$ because $2141 \equiv 2 \hspace{1mm} (\text{mod } 3)$ and every non-zero residue class is a cube modulo a prime that is equivalent to $2 (\text{mod } 3)$. 
\end{enumerate}
From each of the above examples $(1)$ and $(2)$, using Corollary \ref{corexponentiation}, we obtain $|\big(\mathbb{F}_{3}\setminus\{0\}\big)^{4}| = 16$ total examples of intersective polynomials.

\subsection{q = 5.}
\begin{corollary}
Let $p_{1}, p_{2}$ be two distinct primes different from $5$. Then, the polynomial \[f(x) = (x^{5} - p_{1})(x^{5} - p_{2}) (x^{5} - p_{1}p_{2}) (x^{5} - p_{1}^{2}p_{2}) (x^{5} - p_{1}p_{2}^{2}) (x^{5} - p_{1}^{4}p_{2})\] is intersective if and only if the following conditions are satisfied:
\begin{itemize}
    \item Some element of the set \[\{p_{1}, p_{2}, p_{1}p_{2}, p_{1}^{2}p_{2}, p_{1}p_{2}^{2}, p_{1}^{4}p_{2}\}\] is a fifth power modulo $3125$. \vspace{1mm}
    
    \item $p_{1}$ is a fifth power modulo $p_{2}$, and $p_{2}$ is a fifth power modulo $p_{1}$.
\end{itemize}\label{cor2}
\end{corollary}

\begin{proof}
Analogous to that in Corollary \ref{cor1}, the two conditions listed in the statement of the corollary are exactly the conditions $(2)$ and $(3)$ in the statement of Theorem \ref{mainresult}, for $5$-free integers \[a_{1} = p_{1}, a_{2} = p_{2}, a_{3} = p_{1}p_{2}, a_{4} = p_{1}^{2}p_{2}, a_{5} = p_{1}p_{2}^{2} \text{ and } a_{6} = p_{1}^{4}p_{2}.\] Therefore, it suffices to show that $a_{1}, a_{2}, \ldots, a_{6}$ satisfy the condition $(1)$ of Theorem \ref{mainresult}. $p_{1}$ and $p_{2}$ are the only primes dividing $\prod_{j=1}^{6} a_{j}$ and we have,
\begin{multline*}
    \\ \nu_{11} = 1, \nu_{12} = 0, \nu_{13} = 1, \nu_{14} = 2, \nu_{15} = 1, \nu_{16} = 4 \\ \nu_{21} = 0, \nu_{22} = 1, \nu_{23} = 1, \nu_{24} = 1, \nu_{25} = 2, \nu_{26} = 1. \\
\end{multline*}
Therefore, the set of hyperplanes in $\mathbb{F}_{5}^{2}$ associated with $\{a_{1}, a_{2}, \ldots, a_{6}\}$ is \[H_{2}:= \Big\{ x_{1} = 0, x_{2} = 0, x_{1} + x_{2} = 0, 2x_{1} + x_{2} = 0, x_{1} + 2x_{2} = 0 \text{ and } 4x_{1} + x_{2} = 0 \Big\}.\] One can easily see that every $(x_{1}, x_{2}) \in\mathbb{F}_{5}^{2}$ satisfies one of the equation in $H_{2}$. In other words, the set $H_{2}$ forms a linear covering of $\mathbb{F}_{5}^{2}$. The corollary is established.
\end{proof}

\end{document}